\documentclass[reqno]{amsart}
\usepackage{mathrsfs}
\usepackage{color}
\usepackage{amsmath}
\usepackage{amsfonts}
\usepackage{amssymb}
\usepackage{graphicx}%
\usepackage{hyperref}

 \newtheorem{theorem}{Theorem}[section]
 \newtheorem{corollary}{Corollary}[section]
 \newtheorem{lemma}{Lemma}[section]
 \newtheorem{proposition}{Proposition}[section]

 \newtheorem{remark}{Remark}[section]

 \numberwithin{equation}{section}

\begin{document}

\title[strong openness and optimal $L^{2}$ extension]
 {strong openness of multiplier ideal sheaves
 and optimal $L^{2}$ extension}

\author{Qi'an Guan}
\address{Qi'an Guan: School of Mathematical Sciences,
and Beijing International Center for Mathematical Research,
Peking University, Beijing, 100871, China.}
\email{guanqian@math.pku.edu.cn}
\author{Xiangyu Zhou}
\address{Xiangyu Zhou: Institute of Mathematics, AMSS, and Hua Loo-Keng Key Laboratory of Mathematics, Chinese Academy of Sciences, Beijing, China}
\email{xyzhou@math.ac.cn}

\thanks{The authors were partially supported by NSFC-11431013. The first author was supported by
NSFC-11522101}



\date{\today}

\dedicatory{}

\commby{}


\begin{abstract}
In this note,
we reveal that our solution of Demailly's strong openness conjecture implies a matrix version
of the conjecture;
our solutions of two conjectures of Demailly-Koll\'{a}r and Jonsson-Mustat\u{a} implies
the truth of twisted versions of the strong openness conjecture;
our optimal $L^{2}$ extension implies Berndtsson's
positivity of vector bundles associated to holomorphic fibrations over a unit disc.
\end{abstract}

\maketitle

\section{Introduction}

The multiplier ideal sheaves giving an invariant of plurisubharmonic singularities play an important role in several complex variables and complex geometry.
Various important and fundamental properties about the multiplier ideal sheaves have been established, e.g.,
coherence, integrally closedness, Nadel vanishing theorem, the restriction formula and subadditivity property,
the strong openness property (i.e. solution of Demailly's strong openness conjecture), and other important properties. There have been found many interesting applications of these properties.

The Ohsawa-Takegoshi $L^2$ extension and its optimal version was used in the
the proofs of some above properties and their applications,
e.g. the restriction formula and subadditivity property,
the strong openness property, and some other properties and applications.

\

In this note, we'll give some further consequences of our recent works about the strong openness of the multiplier ideal sheaves and optimal $L^2$ extension.

\subsection{A matrix version of Demailly's strong openness conjecture}

Let $X$ be a complex manifold with dimension $n$ and $\varphi$ be a
plurisubharmonic function on $X$ (see \cite{kisel,Si}).
The multiplier ideal sheaf $\mathcal{I}(\varphi)$ is defined to
be the sheaf of germs of holomorphic functions $f$ such that
$|f|^{2}e^{-2\varphi}$ is locally integrable (see \cite{Nadel90,tian87,siu96,siu00,demailly-note2000,siu05,demailly2010}). The basic properties of multiplier ideal sheaves include:
$\mathcal{I}(\varphi)$ is a coherent analytic and integrally closed sheaf and satisfies the Nadel vanishing theorem.

Let
$$\mathcal{I}_{+}(\varphi):=\cup_{\varepsilon>0}\mathcal{I}((1+\varepsilon)\varphi).$$

In \cite{GZopen-a}, we proved Demailly's strong openness conjecture posed in \cite{demailly-note2000,demailly2010} (related effectiveness result see \cite{GZopen-effect}).
\begin{theorem}
\label{thm:open_GZ}\cite{GZopen-a}
$\mathcal{I}_{+}(\varphi)=\mathcal{I}(\varphi)$ holds.
\end{theorem}

This means that the multiplier ideal sheaf has strong openness property. As immediate applications, several questions are solved in \cite{GZopen-c},
such as, a problem about the existence of an analytic weight in a multiplier ideal sheaf,
a conjecture about a more general vanishing theorem than Nadel's posed by Demailly (see \cite{Cao12}),
a conjecture posed by Demailly-Ein-Lazarsfeld (see \cite{DEL00} and \cite{Lazar04I}),
and a conjecture posed by Boucksom-Favre-Jonsson (see \cite{BFJ08}), etc.

Recently, we \cite{GZopen-lelong} characterize the multiplier ideal sheaves with psh weights of Lelong number one
by using the strong openness property of the multiplier ideal sheaf (Theorem \ref{thm:open_GZ} (dimension two case was solved in \cite{B-M,FM05j})).
When Lelong number is smaller than one, it was proved by Skoda that the multiplier ideal sheaves are trivial (\cite{skoda72}, see also \cite{demailly-book,demailly2010}).

When $\mathcal{I}(\varphi)$ is trivial,
Demailly's strong openness conjecture degenerates to
the openness conjecture posed in \cite{D-K01} which was proved by Berndtsson \cite{berndtsson13} (dimension two case was proved by Favre and Jonsson \cite{FM05j}).

After the strong openness conjecture was proved,
Ohsawa was invited by the second author to the  Institute of Mathematics in Chinese Academy of Sciences in January 2014 and gave three lectures.
During his lectures, Ohsawa asked the following matrix version of the strong openness conjecture:

\textbf{Question}: Let $(F_{i,j})_{1\leq i\leq k,1\leq j\leq l}$ $(k\leq l)$ be matrix such that $F_{i,j}$ are holomorphic functions on $X$
and
$$\int_{\Delta^{n}}\det((F_{i,j})\times (\overline{F_{i,j}})^{t}) e^{-\varphi}d\lambda_{n}<+\infty,$$
does there
exist a number $p>1$, such that
$$\int_{\Delta^{n}_{r}}\det((F_{i,j})\times (\overline{F_{i,j}})^{t})e^{-p\varphi}d\lambda_{n}<+\infty,$$
where $r\in(0,1)$?

Note that, by using the Cauchy-Binet formula,
\begin{equation}
\label{equ:20140219}
\begin{split}
\det((F_{i,j})\times (\overline{F_{i,j}})^{t})
&=\sum_{|J|=k}\det((F_{i,j})_{1\leq i\leq k,j\in J}\times (\overline{F_{i,j}})_{1\leq i\leq k,j\in J}^{t})
\\&=\sum_{|J|=k}\det((F_{i,j})_{1\leq i\leq k,j\in J})\times\det((\overline{F_{i,j}})_{1\leq i\leq k,j\in J}^{t})
\\&=\sum_{|J|=k}\det((F_{i,j})_{1\leq i\leq k,j\in J})\times\overline{\det((F_{i,j})_{1\leq i\leq k,j\in J})}
\end{split}
\end{equation}
where $J\subset\{1,\cdots,l\}$.
Then the above question degenerates to the strong openness conjecture, which can be answered by our solution of
the strong openness conjecture.
\begin{proposition}
\label{rem:open_matrix}
The matrix version of the strong openness conjecture holds.
\end{proposition}

\subsection{Twisted versions of Demailly's strong openness conjecture}

Let $I$ be an ideal of $\mathcal{O}_{\Delta^{n},o}$, which is generated by $\{f_{j}\}_{j=1,\cdots,l}$.
Denote by $$\log|I|:=\log\max_{1\leq j\leq l}|f_{j}|,$$
the jumping number is defined to be $c_{o}^{I}(\varphi)=\sup\{c\geq0:|I|^{2}e^{-2c\varphi}$ is $L^1$ on a neighborhood of $o\}$ (\cite{JM13}).
One can check that $\nu(\varphi,o)=0\Leftrightarrow c_{o}(\varphi)=+\infty \Leftrightarrow c_{o}^{I}(\varphi)=+\infty$,
where $I\neq \{0\}$.

It is known that $\mathcal{I}_{+}(\varphi)=\mathcal{I}(\varphi)$ for any psh $\varphi$ is equivalent to the statement that
$|I|^{2}\exp{(-2c_{o}^{I}(\varphi)\varphi)}$ is not locally integrable near $o$ for any $I$ and $\varphi$
satisfying $c_{o}^{I}(\varphi)<+\infty$.
In \cite{GZopen-b,GZopen-effect},
we prove two conjectures posed by Demailly-Koll\'{a}r (see \cite{D-K01}) and Jonsson-Mustat\u{a}
(see \cite{JM12}) respectively by using the following result.

\begin{theorem}
\label{thm:JM_GZ}\cite{GZopen-b,GZopen-effect}
If $c_{o}^{I}(\varphi)<+\infty$,
then
$$\frac{1}{r^2}\int|I|^{2}\chi_{\{c^{I}_{o}(\varphi)\varphi<\log r\}}$$
and
$$\frac{1}{r^2}\int\chi_{\{c^{I}_{o}(\varphi)\varphi-\log|I|<\log r\}}$$
have positive lower bounds independent of $r\in(0,1)$.
\end{theorem}

Dimension two case of the above theorem was proved by Jonsson and Mustata \cite{JM12} ($I$ is trivial see \cite{FM05j}).
The proof of the general case is based on our solution of Demailly's strong openness conjecture \cite{GZopen-a} and our solution of the $L^{2}$ extension problem with optimal estimates \cite{guan-zhou13p,guan-zhou13ap}.

In the present note,
we deduce that Theorem \ref{thm:JM_GZ} implies the the following twisted version of Demailly's strong openness conjecture.

\begin{theorem}
\label{prop:twist_open}
Let $a(t)$ be a positive measurable function on $(-\infty,+\infty)$
such that $a(t)e^{t}$ is strictly increasing and continuous near $+\infty$.
Then the following three statements are equivalent

$(A)$ $a(t)$ is not integrable near $+\infty$;

$(B)$ $a(-2c^{I}_{o}(\varphi)\varphi)\exp{(-2c^{I}_{o}(\varphi)\varphi+2\log|I|)}$ is not integrable near $o$ for any $\varphi$ and $|I|$
satisfying $c^{I}_{o}(\varphi)<+\infty$;

$(C)$ $a(-2c^{I}_{o}(\varphi)\varphi+2\log|I|)\exp{(-2c^{I}_{o}(\varphi)\varphi+2\log|I|)}$ is not integrable near $o$ for any $\varphi$ and $|I|$
satisfying $c^{I}_{o}(\varphi)<+\infty$.
\end{theorem}

When one takes $a\equiv1$, both of $B$ and $C$ degenerate to strong openness of multiplier ideal sheaves.

When one takes $a(t)=\frac{1}{t^{\alpha}}$ and $\alpha>0$, it is clear that $B$ holds if and only if $\alpha\in(0,1]$,
i.e.,
$\frac{1}{|\varphi|^{\alpha}}|I^{2}|e^{-2c^{I}_{o}(\varphi)\varphi}$ is not locally integrable if and only if $\alpha\in(0,1]$
(the case when $I=\mathcal{O}$ was expected in \cite{Chen2016}).

\subsection{The optimal $L^{2}$ extension and applications}
Now let's recall some notations in \cite{ohsawa5}. Let $M$ be a
complex $n-$dimensional manifold, and $S$ be a closed complex
subvariety of $M$. Let $dV_{M}$ be a continuous volume form on $M$.
We consider a class of upper-semi-continuous function $\Psi$ from
$M$ to the interval $[-\infty,A)$, where $A\in(-\infty,+\infty]$,
such that

$1)$ $\Psi^{-1}(-\infty)\supset S$, and $\Psi^{-1}(-\infty)$ is a closed subset of $M$;

$2)$ If $S$ is $l-$dimensional around a point $x\in S_{reg}$
($S_{reg}$ is the regular part of $S$), there exists a local
coordinate $(z_{1},\cdots,z_{n})$ on a neighborhood $U$ of $x$ such
that $z_{l+1}=\cdots=z_{n}=0$ on $S\cap U$ and
$$\sup_{U\setminus S}|\Psi(z)-(n-l)\log\sum_{l+1}^{n}|z_{j}|^{2}|<\infty.$$

The set of such polar functions $\Psi$ will be denoted by
$\#_{A}(S)$.

For each $\Psi\in\#_{A}(S)$, one can associate a positive measure
$dV_{M}[\Psi]$ on $S_{reg}$ as the minimum element of the partially
ordered set of positive measures $d\mu$ satisfying

$$\int_{S_{l}}fd\mu\geq\limsup_{t\to\infty}\frac{2(n-l)}
{\sigma_{2n-2l-1}}\int_{M}fe^{-\Psi}\mathbb{I}_{\{-1-t<\Psi<-t\}}dV_{M}$$
for any nonnegative continuous function $f$ with $supp
f\subset\subset M$, where $\mathbb{I}_{\{-1-t<\Psi<-t\}}$ is the
characteristic function of the set $\{-1-t<\Psi<-t\}$. Here denote
by $S_{l}$ the $l$-dimensional component of $S_{reg}$, denote by
$\sigma_{m}$ the volume of the unit sphere in $\mathbb{R}^{m+1}$.

Let $M$ be a complex manifold with a continuous volume form
$dV_{M}$, and $S$ be a closed complex subvariety of $M$.
We call a pair
$(M,S)$ is an almost Stein pair if $M$ and $S$ satisfy the
following conditions:

There exists a closed subset $X\subset M$ such that:

$(a)$ $X$ is locally negligible with respect to $L^2$ holomorphic
functions, i.e., for any local coordinate neighborhood $U\subset M$
and for any $L^2$ holomorphic function $f$ on $U\setminus X$, there
exists an $L^2$ holomorphic function $\tilde{f}$ on $U$ such that
$\tilde{f}|_{U\setminus X}=f$ with the same $L^{2}$ norm.

$(b)$ $M\setminus X$ is a Stein manifold which intersects with every component of $S$,
such that $S_{sing}\subset X$.

Given $\delta>0$, let $c_{A}(t)$ be a positive function on
$(-A,+\infty)$ $(A\in(-\infty,+\infty))$, which is in
$C^{\infty}((-A,+\infty))$ and satisfies both
$\int_{-A}^{\infty}c_{A}(t)e^{-t}dt<\infty$ and
\begin{equation}
\label{equ:c_A_delta}
\begin{split}
&(\frac{1}{\delta}c_{A}(-A)e^{A}+\int_{-A}^{t}c_{A}(t_{1})e^{-t_{1}}dt_{1})^{2}>\\&c_{A}(t)e^{-t}
(\int_{-A}^{t}(\frac{1}{\delta}c_{A}(-A)e^{A}+\int_{-A}^{t_{2}}c_{A}(t_{1})e^{-t_{1}}dt_{1})
dt_{2}+\frac{1}{\delta^{2}}c_{A}(-A)e^{A}),
\end{split}
\end{equation}
for any $t\in(-A,+\infty)$.

 An easy example of such functions $c_{A}(t)$ is when $c_{A}(t)e^{-t}$ is decreasing with respect to $t$. We have given several interesting  more examples in \cite{guan-zhou13ap} to solve several open questions. We remark that the solutions of the questions are reduced to the choice of such functions $c_{A}(t)$.

In \cite{guan-zhou13ap}, we establish the following optimal $L^{2}$ extension theorem as follows:

\begin{theorem}\label{t:guan-zhou-unify}\cite{guan-zhou13ap}
Let $(M,S)$ be an almost Stein pair,
$h$ be a smooth metric on a holomorphic vector bundle $E$ on $M$ with rank $r$.
Let $\Psi\in \#_{A}(S)\cap C^{\infty}(M\setminus S)$, which satisfies

1), $he^{-\Psi}$ is semi-positive in the sense of Nakano on
$M\setminus (S\cup X)$ ($X$ is as in the definition of condition
$(a,b)$),

2), there exists a continuous function $a(t)$ on $(-A,+\infty]$,
such that $0<a(t)\leq s(t)$ and
$a(-\Psi)\sqrt{-1}\Theta_{he^{-\Psi}}+\sqrt{-1}\partial\bar\partial\Psi$
is semi-positive in the sense of Nakano on $M\setminus (S\cup X)$,
where $$s(t)=\frac{\int_{-A}^{t}(\frac{1}
{\delta}c_{A}(-A)e^{A}+\int_{-A}^{t_{2}}c_{A}(t_{1})e^{-t_{1}}dt_{1})dt_{2}+\frac{1}{\delta^{2}}c_{A}(-A)e^{A}}
{\frac{1}{\delta}c_{A}(-A)e^{A}+\int_{-A}^{t}c_{A}(t_{1})e^{-t_{1}}dt_{1}}.$$
Then for any holomorphic section $f$ of $K_{M}\otimes
E|_{S}$ on $S$ satisfying
\begin{equation}
\label{equ:condition}
\sum_{k=1}^{n}\frac{\pi^{k}}{k!}\int_{S_{n-k}}|f|^{2}_{h}dV_{M}[\Psi]<\infty,
\end{equation}
there
exists a holomorphic section $F$ of $K_{M}\otimes E$ on $M$ satisfying $F = f$ on $ S$ and
\begin{equation}
\label{equ:optimal_delta}
\int_{M}c_{A}(-\Psi)|F|^{2}_{h}dV_{M}
\leq\mathbf{C}(\frac{1}{\delta}c_{A}(-A)e^{A}+\int_{-A}^{\infty}c_{A}(t)e^{-t}dt)
\sum_{k=1}^{n}\frac{\pi^{k}}{k!}\int_{S_{n-k}}|f|^{2}_{h}dV_{M}[\Psi],
\end{equation}
where $c_{A}(t)$ satisfies $c_{A}(-A)e^{A}:=\lim_{t\to -A^{+}}c_{A}(t)e^{-t}<\infty$ and $c_{A}(-A)e^{A}\neq0$,
$\mathbf{C}=1$ (which is optimal).
\end{theorem}

The optimal $L^2$ extension theorem (Theorem \ref{t:guan-zhou-unify} \cite{guan-zhou13ap}) gives unified optimal estimate versions
of various well-known $L^2$ extension theorems in \cite{ohsawa2,manivel93,demailly99,siu00,paun07,berndtsson09,mcneal-varolin,D-P}, etc.
Some interesting relations between the optimal $L^2$ extension and some questions are found, so that the questions are solved in \cite{guan-zhou13ap} by using optimal $L^2$ extension (Theorem \ref{t:guan-zhou-unify}),
such as Suita's conjecture (see \cite{suita} \cite{sario}), L-conjecture (see \cite{yamada98}),
extended Suita conjecture (see \cite{yamada98}), and an open question posed by Ohsawa (see \cite{ohsawa6}), etc.

Let $c_{\infty}(t)$ be a positive function in $C^{\infty}((-\infty,+\infty))$ satisfying
$\int_{-\infty}^{\infty}c_{A}(t)e^{-t}dt<\infty$ and
\begin{equation}
\label{equ:c_A}
\big(\int_{-\infty}^{t}c_{\infty}(t_{1})e^{-t_{1}}dt_{1}\big)^{2}>c_{\infty}(t)e^{-t}
\int_{-\infty}^{t}\int_{-\infty}^{t_{2}}c_{\infty}(t_{1})e^{-t_{1}}dt_{1}dt_{2},
\end{equation}
for any $t\in(-\infty,+\infty)$. This class of functions will be denoted by $\mathcal{C}_{\infty}$.

\begin{remark}\label{condition}(see \cite{guan-zhou13ap})
Assume that $\frac{d}{dt}c_{\infty}(t)e^{-t}>0$
for $t\in (-\infty,a)$,
$\frac{d}{dt}c_{\infty}(t)e^{-t}\leq0$
for $t\in[a,+\infty)$,
and
$\frac{d^{2}}{dt^{2}}\log(c_{\infty}(t)e^{-t})<0$
for $t\in (-\infty,a)$,
where  $a>-\infty$.
Then inequality \ref{equ:c_A} holds.
\end{remark}

Especially, when we take $X$ be a pseudo-convex domain in $\mathbb{C}^{n}$ with coordinates
$(z_{1},\cdots,z_{n})$ and $\Psi=\log|z_{n}|$,
Theorem \ref{t:guan-zhou-unify} degenerates to

\begin{theorem}\label{t:guan-zhou-domain}\cite{guan-zhou13ap}
Let $D\subset\mathbb{C}^n$ be a pseudoconvex domain, $\varphi$ a plurisubharmonic
function on $D$ and $H=\{z_n=0\}$. Then for any holomorphic
function $f$ on $H$ satisfying
$$\int_{H}|f|^{2}e^{-\varphi}dV_{H}<\infty,$$
there exists a holomorphic function $F$ on $D$ satisfying $F = f$ on $H$ and
\begin{eqnarray*}
\int_{D}c_{\infty}(-2\log|z_n|)|F|^{2}e^{-\varphi}dV_{D}
\leq\mathbf{C}\pi\int_{-\infty}^{\infty}c_{\infty}(t)e^{-t}dt\int_{H}|f|^{2}e^{-\varphi}dV_{H},
\end{eqnarray*}
where $\mathbf{C}=1$ (which is optimal).
\end{theorem}

Let $c_{\infty}(t)=e^{-\alpha e^{-t}}$ in Theorem \ref{t:guan-zhou-domain}.
By Remark \ref{condition}, it suffices to find $a\in\mathbb{R}$ such that
\
$(A)$ $\frac{d}{dt}(c_{\infty}(t)e^{-t})>0$ for $(-\infty,a)$;
\
$(B)$ $\frac{d}{dt}(c_{\infty}(t)e^{-t})\leq0$ for $[a,+\infty)$;
\
$(C)$ $\frac{d^2}{dt^2}\log(c_{\infty}(t)e^{-t})<0$ for $(-\infty,+\infty)$.

Denote by $G(t):=e^{-\alpha e^{-t}}e^{-t}=e^{-(\alpha e^{-t}+t)}$, where $\alpha >0$. Then,
$$G'(t)=-(\alpha e^{-t}+t)'G(t)=-(1-\alpha e^{-t})G(t)=(\alpha e^{-t}-1)G(t).$$
Take $a=\log\alpha$, then (A) and (B) holds. Note that
$$\log(c_{\infty}(t)e^{-t})=-(\alpha e^{-t}+t),\ \frac{d^2}{dt^2}(-\alpha(e^{-t}+t))=-\alpha e^{-t}<0$$
for $t\in(-\infty,+\infty)$, then (C) holds.
Then we obtain the following corollary of Theorem \ref{t:guan-zhou-domain}

\begin{corollary}
\label{coro:GZ-domain}
Let $D\subset\mathbb{C}^n$ be a pseudoconvex domain, and let $\varphi$ be a plurisubharmonic
function on $D$ and $H=\{z_n=0\}$. Then for any holomorphic
function $f$ on $H$ satisfying
$$\int_{H}|f|^{2}e^{-\varphi}dV_{H}<\infty,$$
there exists a holomorphic function $F$ on $D$ satisfying $F = f$ on $H$ and
\begin{eqnarray*}
\int_{D}e^{-\alpha|z_n|^2}|F|^{2}e^{-\varphi}dV_{D}
\leq\frac{\pi}{\alpha}\int_{H}|f|^{2}e^{-\varphi}dV_{H}.
\end{eqnarray*}
\end{corollary}

The above result was listed as an open problem by Ohsawa in Winter School of Sanya School in Complex Analysis and Geometry in 2016.
In his recent paper "On the extension of $L^2$ holomorphic functions VIII -- a remark on a theorem of Guan and Zhou", Ohsawa recognized that Corollary \ref{coro:GZ-domain} could be implied by the Main Theorem in [21].
\

Let $p$ be subjective family holomorphic map, from compact manifold $M$ to $\Delta$ (with coordinate $z$) with surjective
differential, and all the fibers $M_{z}=p^{-1}(z)$ are assumed to be compact.
Let $(L,h_{L})$ be a Hermitian line bundle on $M$.
One can define a holomorphic vector bundle $E$ over $\Delta$
with $E_{z}=H^{0}(M_{z},K_{M}|_{M_{z}}\otimes L|_{M_{z}})$,
and the $H^{0}(\Delta,E)=H^{0}(M,K_{M}\otimes L)$.
Let $u$ be a holomorphic section of $E$, and let $||u||:=||u||_{z}=\int_{M_{t}}|\frac{u}{dz}|^{2}_{h_{L}}$ as in \cite{berndtsson09},
which deduces a hermitian metric on $E$.

In \cite{berndtsson09},
Berndtsson establish the following positivity of vector bundles associated to holomorphic fibrations.

\begin{theorem}
\label{thm:bernd09}\cite{berndtsson09}
If the total space $M$ is K\"{a}hler and $L$ is (semi)positive over
$M$, then $(E, ||\cdot||)$ is (semi)positive in the sense of Nakano.
\end{theorem}

In \cite{guan-zhou13ap}, we find that the optimal $L^2$ extension theorem (Theorem
\ref{t:guan-zhou-unify}) implies Berndtsson's theorem on
log-plurisubharmonicity of Bergman kernel \cite{berndtsson98,berndtsson06,berndtsson09}. We remark here that only the optimal estimate could do so.

In the present note, we will reveal that

\begin{proposition}
\label{prop:bern_positive}
Theorem \ref{t:guan-zhou-unify} implies Theorem \ref{thm:bernd09}.
\end{proposition}

\section{Preparations}

The following lemma will be used to prove Theorem \ref{prop:twist_open}

\begin{lemma}
\label{lem:integ}
For any two measurable spaces $(X_{i},\mu_{i})$ and two measurable functions $g_{i}$ on $X_{i}$ respectively ($i\in\{1,2\}$),
if $\mu_{1}(\{g_{1}\geq r^{-1}\})\geq\mu_{2}(\{g_{2}\geq r^{-1}\})$ for any $r\in(0,r_{0}]$, then $\int_{\{g_{1}\geq r_{0}^{-1}\}} g_{1}d\mu_{1}\geq\int_{\{g_{2}\geq r_{0}^{-1}\}} g_{2}d\mu_{2}$.
\end{lemma}

\begin{proof}
Consider the functions $g_{i,m}:=r_{0}^{-1}\chi_{\{g_{i}\geq r_{0}^{-1}\}}+\frac{1}{2^{m}}\sum_{n\geq 1}\chi_{\{g_{i}\geq r_{0}^{-1}+\frac{n}{2^{m}}\}}$,
$m\in\mathbb{N}^{+}$.
One can check that $g_{i,m}$ is increasing with respect to $m$, and convergent to $g_{i}$ on $\{g_{i}\geq r_{0}^{-1}\}$ when $m$ goes to $+\infty$.
As $\mu_{1}(\{g_{1}\geq r^{-1}\})\geq\mu_{2}(\{g_{2}\geq r^{-1}\})$ for any $r\in(0,r_{0}]$,
then
$\int_{\{g_{1}\geq r_{0}^{-1}\}} g_{1,m}d\mu_{1}\geq \int_{\{g_{2}\geq r_{0}^{-1}\}} g_{2,m}d\mu_{2}$ for any $m\in\mathbb{N}^{+}$.
By Levi's Theorem, it follows that $\int_{\{g_{i}\geq r_{0}^{-1}\}} g_{i}d\mu_{i}=\lim_{m\to+\infty}\int_{\{g_{i}\geq r_{0}^{-1}\}} g_{i,m}d\mu_{i}$,
which deduces the present lemma.
\end{proof}

Let $c_{A}(t)=1$, $A=0$, $p:M\to\Delta$ be a projective family over unit disc $\Delta$ with coordinate $z$,
and $L$ be a smooth semipositive line bundle on $M$,
and $\Psi=2p^{*}\log\frac{|z|}{r}$ (similar method in \cite{guan-zhou13ap} implies the compact K\"{a}hler family case).
Then the combination of Theorem \ref{t:guan-zhou-unify} and
the following lemma implies (the semi-positive part of) Proposition \ref{prop:bern_positive}.

\begin{lemma}
\label{lem:Nakano}
Let $(V,h)$ be a trivial hermitian vector bundle rank $m$ over the unit disc $\Delta$ with coordinate $z$.
Assume that for any $u_{0}\in \mathbb{C}^{m}$ and $r>0$,
there exists holomorphic section $u$ of $V|_{\Delta_{r}}$,
such that $u(0)=u_{0}$ and $\frac{1}{\pi r^{2}}\int_{\Delta_{r}}|u|_{h}^{2}\leq |u_{0}|_{h}^{2}$.
Then $(V,h)$ is Nakano semi-positive at $0\in\Delta\subset\mathbb{C}$.
\end{lemma}

\begin{proof}
It is known that there exists a holomorphic frame $(e_{j})_{1\leq j\leq m}$ of $V$ on a neighbourhood of $0\in\Delta\subset\mathbb{C}$,
such that
\begin{equation}
\label{equ:curv}
\langle e_{j},e_{k}\rangle=\delta_{j,k}-c_{j,k}|z|^{2}+O(|z|^{3}),
\end{equation}
where $c_{j,k}$ is the Chern curvature tensor at $0$ (see \cite{demailly-book}, Chapter V, Proposition 12.10).

We prove the present lemma by contradiction: if not, there exists a holomorphic frame $(e_{j})_{1\leq j\leq m}$ on a neighbourhood of $0\in\Delta\subset\mathbb{C}$,
such that $c_{1,1}<0$ $c_{j,k}=0$ for any $j\neq k$ in equality \ref{equ:curv} (by the unitary transformation of $(e_{j})_{1\leq j\leq m}$).

Let $u_{0}$ be $e_{1}(0)$.
Then exists section $u$ on $TV_{\Delta_{r}}$,
such that
\begin{equation}
\label{equ:curv_1}
u=\sum_{1\leq j\leq m}e_{j}f_{j}
\end{equation}
and
\begin{equation}
\label{equ:curv_2}
\frac{1}{\pi r^{2}}\int_{\Delta_{r}}(|f|^{2}-\sum_{1\leq j,k\leq m}(c_{j,k}f_{j}\bar{f}_{k}|z|^{2}+O(|z|^{3})|f|^{2}))
\leq |u_{0}|_{h}^{2}
\end{equation}
where
$f_{j}$ are holomorphic functions on $\Delta_{r}$ and $f_{j}(0)=\delta_{1,j}$,
$|f|^{2}=\sum_{1\leq j\leq m}|f_{j}|^{2}$, and $O(|z|^{3})$ is independent of $r$.
One can choose $r$ small enough such that

$(1)$ $|c_{j,j}|r^{2}\leq\frac{1}{6m}$, which implies $|c_{j,j}f_{j}\bar{f}_{j}||z|^{2}\leq \frac{1}{6m}(|f_{j}|^{2})$,
where $2\leq j,k\leq m$;

$(2)$ $|O(|z|^{3})|\leq -\frac{1}{6}c_{1,1}|z|^{2}$ and $|O(|z|^{3})|\leq \frac{1}{6}$ on $\Delta_{r}$, which implies
$|O(|z|^{3})||f|^{2}\leq -\frac{1}{6}c_{1,1}|z|^{2}|f_{1}|^{2}+\frac{1}{6}\sum_{2\leq j\leq m}|f_{j}|^{2}$
on $\Delta_{r}$

Combining (1) and (2),
one can obtain that
\begin{equation}
\label{equ:curv_3}
\begin{split}
&|f|^{2}-\sum_{1\leq j,k\leq n}(c_{j,k}f_{j}\bar{f}_{k}|z|^{2}+O(|z|^{3})|f|^{2})
\\\geq&
|f|^{2}-c_{1,1}|f_{1}|^{2}|z|^{2}
-\sum_{2\leq j,k\leq m}|c_{j,j}f_{j}\bar{f}_{j}||z|^{2}
-|O(|z|^{3})||f|^{2}
\\\geq&
|f|^{2}-c_{1,1}|f_{1}|^{2}|z|^{2}
-\sum_{2\leq j\leq m}\frac{1}{6m}(|f_{j}|^{2})
-(-\frac{1}{6}c_{1,1}|z|^{2}|f_{1}|^{2}+\frac{1}{6}\sum_{2\leq j\leq m}|f_{j}|^{2})
\\\geq &
|f_{1}|^{2}-\frac{c_{1,1}}{2}|f_{1}|^{2}|z|^{2}.
\end{split}
\end{equation}
As $c_{1,1}<0$,
it follows that
$\frac{1}{\pi r^{2}}\int_{\Delta_{r}}(|f_{1}|^{2}-\frac{c_{1,1}}{2}|f_{1}|^{2}|z|^{2})>|f_{1}(0)|^{2}=|u_{0}|_{h}^{2}$.
Combining with inequality \ref{equ:curv_2} and \ref{equ:curv_3},
we obtain
$\frac{1}{\pi r^{2}}\int_{\Delta_{r}}|u|_{h}^{2}> |u_{0}|_{h}^{2}$,
which contradicts $\frac{1}{\pi r^{2}}\int_{\Delta_{r}}|u|_{h}^{2}\leq |u_{0}|_{h}^{2}$.
Then we obtain the present lemma.
\end{proof}

Lemma \ref{lem:Nakano} implies the following

\begin{lemma}
\label{lem:strict_nakano}
Let $(V,h)$ be a hermitian vector bundle rank $m$ on the unit disc $\Delta$ with coordinate $z$.
Assume that there exists $\varepsilon>0$
such that $(V,he^{\varepsilon|z|^{2}})$ is Nakano semipositive at $0\in\Delta\subset\mathbb{C}$.
Then $(V,h)$ is strictly Nakano positive at $0\in\Delta\subset\mathbb{C}$.
\end{lemma}

Let $c_{A}(t)=1$, $A=0$, $p:M\to\Delta$ be a projective family over unit disc $\Delta$ with coordinate $z$,
and $L$ be a smooth positive line bundle on $M$,
and $\Psi=2p^{*}\log\frac{|z|}{r}$ (similar method in \cite{guan-zhou13ap} implies the compact K\"{a}hler family case).
Then the combination of Theorem \ref{t:guan-zhou-unify} and
Lemma \ref{lem:strict_nakano} implies (the positive part of) Proposition \ref{prop:bern_positive}.

\section{Proof of Theorem \ref{prop:twist_open}}

The proof of the result is divided into three steps.
It suffices to consider the case that $a(t)e^{t}$ is continuous.

\

\textbf{Step 1}. We will prove $B\Rightarrow A$ and $C \Rightarrow A$.

Consider $I=1$ and $\varphi=\log|z_{1}|$ on the unit polydisc $\Delta^{n}\subset\mathbb{C}$,
note that $c_{o}(\log|z_{1}|)=1$ and
\begin{equation}
\begin{split}
\int_{\Delta^{n}_{r_{0}}}a(-2\log|z_{1}|)\frac{1}{|z_{1}|^{2}}
&=
(\pi r_{0}^{2})^{n-1}\int_{\Delta_{r_{0}}}a(-2\log|z_{1}|)\frac{1}{|z_{1}|^{2}}
\\&=
(\pi r_{0}^{2})^{n-1}2\pi\int_{[0,r_{0}]}a(-2\log r)r^{-2}rdr
\\&=
(\pi r_{0}^{2})^{n-1}\pi\int_{[-2\log r_{0},+\infty]}a(t)dt,
\end{split}
\end{equation}
then we obtain $B\Rightarrow A$ and $C \Rightarrow A$.

\

\textbf{Step 2}. We will prove $A\Rightarrow B$.

Let $X_{1}$ be a small neighborhood $U$ near $o$, and $X_{2}=(0,1]$.
Let $\mu_{1}(\cdot)=\int_{\cdot}|I|^{2}d\lambda$,
where $\lambda$ be the Lebesgue measure on $U\subset\mathbb{C}^{n}$.
Let $\mu_{2}$ be the Lebesgue measure on $X_{2}$.
Let $Y_{r}:=\{\exp(-2c^{I}_{o}(\varphi)\varphi)\geq r^{-1}\}$.
Theorem \ref{thm:JM_GZ} shows that
there exists positive constant $C$
such that $\mu_{1}(Y_{r})\geq C r$
holds for any $r\in(0,1]$.

Let
$g_{1}=a(-2c^{I}_{o}(\varphi)\varphi)\exp{(-2c^{I}_{o}(\varphi)\varphi)}$ and $g_{2}(x)=a(-\log x+\log C)C x^{-1}$.
As $a(t)e^{t}$ is increasing near $+\infty$,
then it follows that $g_{1}\geq a(-\log r)r^{-1}$ on $Y_{r}$,
which gives
\begin{equation}
\label{equ:AC1}
\mu_{1}(\{g_{1}\geq a(-\log r)r^{-1}\})\geq \mu_{1}(Y_{r})\geq C r.
\end{equation}

As $a(t)e^{t}$ is convergent to $+\infty$ $(t\to +\infty)$,
then $a(-\log r)r^{-1}$ convergent to $+\infty$ $(r\to0+0)$.
As $a(t)e^{t}$ is strictly increasing near $+\infty$ and $a(t)e^{t}$ is convergent to $+\infty$ $(t\to +\infty)$,
then $\{g_{2}\geq a(-\log r)r^{-1}\}=\{-\log x+\log C\geq -\log r\}=\{0<x\leq C r\}$ for small enough $r>0$,
which implies
\begin{equation}
\label{equ:AC2}
\mu_{2}(\{g_{2}\geq a(-\log r)r^{-1}\})= \mu_{2}(\{0<x\leq Cr\})=Cr,
\end{equation}
 for small enough $r>0$.

Inequality \ref{equ:AC1} and \ref{equ:AC2} gives that
$$\mu_{1}(\{g_{1}\geq a(-\log r)r^{-1}\})\geq\mu_{2}(\{g_{2}\geq a(-\log r)r^{-1}\})$$
for any $r$ small enough.
Using the continuity of $a(-\log r)r^{-1}$ and $a(-\log r)r^{-1}$ convergence to $+\infty$ $(r\to0+0)$,
we obtain that $\mu_{1}(\{g_{1}\geq r^{-1}\})\geq\mu_{2}(\{g_{2}\geq r^{-1}\})$ for any $r>0$ small enough.
By Lemma \ref{lem:integ}, we obtain $A \Rightarrow B$.

\

\textbf{Step 3}. We will prove $A\Rightarrow C$.

Let $X_{1}$ be a small neighborhood near $o\in\mathbb{C}^{n}$, and $X_{2}=(0,1]$,
and $\mu_{1}$ and $\mu_{2}$ be the Lebesgue measure on $X_{1}$ and $X_{2}$ respectively.
Let $Y_{r}:=\{\exp(-2c^{I}_{o}(\varphi)\varphi+2\log|I|)\geq r^{-1}\}$.
Theorem \ref{thm:JM_GZ} shows that
there exists positive constant $C$
such that $\mu_{1}(Y_{r})\geq C r$
holds for any $r\in(0,1]$.

Let
$$g_{1}=a(-2c^{I}_{o}(\varphi)\varphi+2\log|I|)\exp{(-2c^{I}_{o}(\varphi)\varphi+2\log|I|)}$$
and
$$g_{2}(x)=a(-\log x+\log C)C x^{-1}.$$
As $a(t)e^{t}$ is increasing near $+\infty$,
then it follows that $g_{1}\geq a(-\log r)r^{-1}$ on $Y_{r}$,
which gives
\begin{equation}
\label{equ:AB1}
\mu_{1}(\{g_{1}\geq a(-\log r)r^{-1}\})\geq \mu_{1}(Y_{r})\geq C r.
\end{equation}

As $a(t)e^{t}$ is convergent to $+\infty$ $(t\to +\infty)$,
then $a(-\log r)r^{-1}$ convergent to $+\infty$ $(r\to0+0)$.
As $a(t)e^{t}$ is strictly increasing near $+\infty$ and $a(t)e^{t}$ is convergent to $+\infty$ $(t\to +\infty)$,
then $\{g_{2}\geq a(-\log r)r^{-1}\}=\{-\log x+\log C\geq -\log r\}=\{0<x\leq C r\}$ for small enough $r>0$,
which implies
\begin{equation}
\label{equ:AB2}
\mu_{2}(\{g_{2}\geq a(-\log r)r^{-1}\})= \mu_{2}(\{0<x\leq C r\})=Cr,
\end{equation}
for small enough $r>0$.

Inequalities \ref{equ:AB1} and \ref{equ:AB2} give that
$$\mu_{1}(\{g_{1}\geq a(-\log r)r^{-1}\})\geq\mu_{2}(\{g_{2}\geq a(-\log r)r^{-1}\})$$
for any $r>0$ small enough.
Using the continuity of $a(-\log r)r^{-1}$ and $a(-\log r)r^{-1}$ convergence to $+\infty$ $(r\to0+0)$,
we obtain that $\mu_{1}(\{g_{1}\geq r^{-1}\})\geq\mu_{2}(\{g_{2}\geq r^{-1}\})$ for any $r>0$ small enough.
By Lemma \ref{lem:integ}, we obtain $A \Rightarrow C$.

\bibliographystyle{references}
\bibliography{xbib}

\end{document}